\documentclass[twoside,11pt]{article}

\usepackage{graphicx, subfigure}
\usepackage[top=1in,bottom=1in,left=1in,right=1in]{geometry}
\usepackage[sort&compress]{natbib} 
\usepackage{amsfonts, amsmath, amssymb, amsthm, constants, bbm}
\usepackage{MnSymbol}
\usepackage{mathtools}
\usepackage[mathscr]{eucal}
\usepackage{booktabs}
\usepackage{float}
\usepackage{color}
\definecolor{darkred}{RGB}{100,0,0}
\definecolor{darkgreen}{RGB}{0,100,0}
\definecolor{darkblue}{RGB}{0,0,150}

\usepackage{hyperref}
\hypersetup{colorlinks=true, linkcolor=darkred, citecolor=darkgreen, urlcolor=darkblue}
\usepackage{url}

\def\R{\mathscr{R}}
\def\Null{\cI_0}

\def\H{\bbH}

\def\fdr{\textsc{fdr}}
\def\fdp{\textsc{fdp}}
\def\fnr{\textsc{fnr}}
\def\fnp{\textsc{fnp}}

\newtheorem{thm}{Theorem}

\newtheorem{lem}{Lemma}

\theoremstyle{remark}

\newtheorem{rem}{Remark}

\def\beq{\begin{equation}} 
\def\eeq{\end{equation}}
\def\beqn{\beq\notag}
\def\Bitem{\begin{itemize}\setlength{\itemsep}{.2in}}
\def\bitem{\begin{itemize}\setlength{\itemsep}{.05in}}
\def\eitem{\end{itemize}}
\def\Benum{\begin{enumerate}\setlength{\itemsep}{.2in}}
\def\benum{\begin{enumerate}\setlength{\itemsep}{.05in}}
\def\eenum{\end{enumerate}}
\def\bmult{\begin{multline*}}
\def\emult{\end{multline*}}
\def\bcenter{\begin{center}}
\def\ecenter{\end{center}}
\def\bframe{\begin{frame}}
\def\eframe{\end{frame}}

\newcommand{\thmref}[1]{Theorem~\ref{thm:#1}}

\newcommand{\lemref}[1]{Lemma~\ref{lem:#1}}
\newcommand{\secref}[1]{Section~\ref{sec:#1}}
\newcommand{\figref}[1]{Figure~\ref{fig:#1}}
\newcommand{\tabref}[1]{Table~\ref{tab:#1}}

\DeclareMathOperator*{\argmax}{arg\, max}

\def\cA{\mathcal{A}}

\def\cI{\mathcal{I}}

\def\cL{\mathcal{L}}

\def\cN{\mathcal{N}}

\def\bP{\mathbf{P}}

\def\bbH{\mathbb{H}}
\def\bbI{\mathbb{I}}

\def\bbR{\mathbb{R}}

\newcommand{\E}{\operatorname{\mathbb{E}}}

\def\iid{\stackrel{\rm iid}{\sim}}

\def\eps{\varepsilon}

\def\iff{\ \Leftrightarrow \ }

\def\1{\mathbbm{1}}
\newcommand{\IND}[1]{\bbI\{ #1 \}}

\definecolor{purple}{rgb}{0.4,.1,.9}
\definecolor{new}{rgb}{0.5,0.1,0.1}

\newcommand\blfootnote[1]{%
  \begingroup
  \renewcommand\thefootnote{}\footnote{#1}%
  \addtocounter{footnote}{-1}%
  \endgroup
}

\begin{document}
\thispagestyle{empty}

\title{A Scan Procedure for Multiple Testing}
\author{Shiyun Chen \and Andrew Ying \and Ery Arias-Castro}
\date{}
\maketitle

\blfootnote{The authors are with the Department of Mathematics, University of California, San Diego, USA.  Contact information is available \href{http://www.math.ucsd.edu/people/graduate-students/}{here} and \href{http://math.ucsd.edu/\~eariasca}{here}.}

\begin{abstract}
In a multiple testing framework, we propose a method that identifies the interval with the highest estimated false discovery rate of P-values and rejects the corresponding null hypotheses.  Unlike the Benjamini-Hochberg method, which does the same but over intervals with an endpoint at the origin, the new procedure `scans' all intervals.  In parallel with \citep*{storey2004strong}, we show that this scan procedure provides strong control of asymptotic false discovery rate.  In addition, we investigate its asymptotic false non-discovery rate, deriving conditions under which it outperforms the Benjamini-Hochberg procedure.  For example, the scan procedure is superior in power-law location models. 

\end{abstract}


\section{Introduction} \label{sec:intro}
Multiple testing problems arise in a wide range of applications, and are most acute in contexts where data are large and complex, and where standard data analysis pipelines involve performing a large number of tests.  
\cite{benjamini1995controlling} proposed to control the false discovery rate (FDR) as a much less conservative criterion than the family-wise error rate (FWER).   They also proposed a method (referred to as the BH method henceforth) for achieving this under some conditions, such as independence of the P-values.
Since then, FDR controlling methods have been proposed and in turn adopted by practitioners faced with large-scale testing problems.  
Although a number of variants have been proposed, most of these methods are also based on computing a threshold based on the P-values and rejecting the null hypotheses corresponding to P-values below that threshold \citep{genovese2004stochastic, storey2002direct, storey2004strong}.
See \cite{roquain2011type} for a survey. 

A threshold approach to multiple testing is natural stemming from the fact that the smaller a P-value is, the more evidence it provides against the null hypothesis being tested.  However, we argue that this is not so obvious in the context of multiple testing, particularly in harder cases where the alternatives are not easily identified and in which most of the the smallest P-values come from true null hypotheses.  
This was already understood by \cite{chi2007performance}, who proposed a complex method which may be roughly described as applying the BH method at multiple locations in the unit interval, each location playing the role of the origin.  The result is a rejection region\footnote{ In context of multiple testing, a rejection region is a subset of the unit interval which identifies the P-values whose null hypotheses are to be rejected.} made of possibly multiple intervals.
  
In the present paper we propose a simpler approach based on the longest interval whose estimated FDR is below the prescribed level.  Compared to \citep{chi2007performance}, the method is simpler and is already shown to outperform the BH method in some settings of potential interest, such as in power-law location models.  The method is simple and intuitive, and can be seen as a direct extension of the approach of \cite{storey2002direct}.
It thus presents a sort of minimal working example where looking beyond threshold methods can be beneficial.  

Scanning over intervals is a common procedure for detecting areas of interest in a point process at least since the work of \cite{naus1965distribution}.  In this context, and its extension to discrete signals, the main task has been to test for homogeneity, and some articles have tackled such situations from a multiple testing angle \citep{siegmund2011false, picard2017continuous, benjamini2007false, caldas2006controlling, pacifico2007scan, perone2004false}.
While these papers aim at controlling the FDR when scanning spatiotemporal data, here we consider a standard multiple testing situation with a priori no spatiotemporal structure, and offer scanning as a way to generalize and potentially improve upon threshold procedures.

\subsection{Framework}
We consider a setting where we test $n$ null hypotheses, denoted by $\H_1, \dots, \H_n$.  The test for $\H_i$ yields a P-value, denoted as $P_i$, and we assume (for simplicity) that these P-values are independent.  
In this context, a multiple testing procedure $\R$ takes the P-values, $\bP = (P_1, \dots, P_n)$, and returns a subset of indices representing the null hypotheses that the procedure rejects. 
\tabref{outcomes} describes the outcome when applying some significance rule in such a setting and defines some necessary notations.
We will let $\Null \subset [n]$ index the true null hypotheses.
 
Given such a procedure $\R$, the false discovery rate is defined as the expected value of the false discovery proportion \citep{benjamini1995controlling}, namely
\beqn
\fdr(\R) = \E(\fdp(\R)), \quad \text{where } \fdp(\R) := \frac{V_\R}{R_\R \vee 1},
\eeq
While the FDR of a multiple testing procedure is analogous to the type I error rate of a test procedure, the false non-discovery rate (FNR) plays the role of type II error rate and is here defined as the expected value of the false non-discovery proportion\footnote{ This definition is different from that of \cite{genovese2002operating}.}, namely
\beqn
\fnr(\R) = \E(\fnp(\R)), \text{where } \quad \fnp(\R) := \frac{T_\R}{n_1}.
\eeq
Note that this definition is different that than introduced in \cite{genovese2002operating}, although there is no substantial difference.

\begin{table}[t!]
	\caption{This table summarizes the outcome of applying a multiple testing procedure $\R$ to a particular situation involving $n$ null hypotheses.}
	\medskip
	\centering
	\begin{tabular}{c c c c} 
		& accept null & reject null & total \\ [0.5ex]
		\midrule
		null true & $U_\R$ & $V_\R$ & $n_0$ \\
		null false & $T_\R$ & $S_\R$ & $n_1$\\
		total & $W_\R$ & $R_\R$ & $n$ \\[0.5ex]
		\hline
	\end{tabular}
	\label{tab:outcomes}
\end{table}

\subsection{Threshold procedures}
Threshold procedures are of the form
\beq\label{threshold}
\R(\bP) = \big\{i : P_i \le \tau(\bP)\big\},
\eeq
where $\tau$ is some (measurable) function with values in $[0,1]$.  As we stated earlier, most multiple testing procedures are of this form, including the BH method.  Specifically, following \cite{storey2002direct}, we may describe the BH method as follows.
For  $0 \le t \le 1$, define the following quantities (see \tabref{outcomes}), 
\beqn
V(t) = \# \{i \in \Null: P_i \le t \}, \quad
S(t) = \# \{i \notin \Null: P_i \le t \},
\eeq
and
\beqn
R(t) = V(t) + S(t) = \# \{i: P_i \le t \},
\eeq
as well as 
\beqn
\fdr(t) = \E \bigg(\frac{V(t)}{R(t) \vee 1}\bigg).
\eeq
This is the FDR of the procedure with rejection region $[0, t]$.  It is estimated by replacing $V(t)$ by $n t$, justified by the fact that $\E(V(t)) \le n_0 t \le n t$.  (The first inequality is an equality when all the null P-values are uniformly distributed in $[0,1]$.) 
This yields
\beqn
\widehat{\fdr}(t) = \frac{n t}{R(t) \vee 1},
\eeq
and the BH method may be defined via the threshold,
\beqn
\hat\tau_\diamond = \max\big\{ t : \widehat{\fdr}(t) \le \alpha\},
\eeq
if it is desired to control the FDR at $\alpha \in (0,1)$.

\subsection{Scan procedures} \label{sec:scan}
Effectively, threshold procedures examine intervals of the form $[0, t]$, where $t \in [0,1]$.  
We extend this family of procedures by considering all possible intervals, thus defining scan procedures as those of the form
\beq\label{interval}
\R(\bP) = \{i : \sigma(\bP) \le P_i \le \tau(\bP)\},
\eeq
where $\sigma$ and $\tau$ are some (measurable) functions with values in $[0,1]$ and such that $\sigma \le \tau$ pointwise.  
Within this family of procedures, we define a specific procedure in analogy with the definition of the BH method given above.

For  $0 \le s \le t \le 1$, define the following quantities (see \tabref{outcomes}),
\beqn
V(s,t) = \# \{i \in \Null: s \le P_i \le t \}, \quad
S(s,t) = \# \{i \notin \Null: s \le P_i \le t \},
\eeq
and
\beqn
R(s,t) = V(s,t) + S(s,t) = \# \{i: s \le P_i \le t \},
\eeq
as well as 
\beqn
\fdr(s,t) = \E(\fdp(s,t)), \quad \text{where } \fdp(s,t) := \frac{V(s,t)}{R(s,t) \vee 1}.
\eeq
This is the FDR of the procedure with rejection region $[s, t]$, which we estimate by replacing $V(s,t)$ with $n (t-s)$, which bounds its expectation, obtaining
\beqn
\widehat{\fdr}(s,t) = \frac{n(t-s)}{R(s,t) \vee 1},
\eeq
and our scan procedure is defined via the interval
\beq\label{scan}
(\hat\sigma, \hat\tau) = \argmax\big\{t-s: \widehat{\fdr}(s,t) \le \alpha\big\},
\eeq
assuming, again, that we desire to control the FDR at $\alpha \in (0,1)$.  If there are several maximizing intervals, we choose the left-most interval.

\begin{rem}
By construction, relying on basic properties of the function $\widehat{\fdr}$, we have that $\hat\sigma$ and $\hat\tau$ correspond to P-values, and 
\beq\label{fdr-alpha}
\widehat{\fdr}(\hat\sigma, \hat\tau) \le \alpha.
\eeq 
\end{rem}

\subsection{Contribution and contents}

In this paper, following \citep{storey2002direct, storey2004strong, genovese2002operating}, we consider an asymptotic setting where the scan procedure just defined is indeed able to control the FDR as desired.  
In the same framework, we also compare, in terms of FNR, the scan procedure with BH procedure, showing that the former is superior to the latter under some specific circumstances, including in power-law location models.

The rest of the paper is organized as follows. 
In \secref{FDR-results} we consider our scan procedure's ability to control the  FDR.  This is established in an asymptotic setting.  
In \secref{FNR-results} we analyze the asymptotic FNR of our scan procedure and compare it with that of the BH procedure.  In particular, we derive sufficient conditions under which the scan procedure outperforms the BH procedure. 
We present the results of numerical experiments in \secref{numerics}.
\secref{discussion} is a brief discussion section. 
All proofs are gathered in \secref{proofs}.

\section{False discovery rate} \label{sec:FDR-results}
In this section we examine how the scan procedure defined in \secref{scan} is able to control the false discovery rate (FDR).
We start with the following result, which shows that $\widehat{\fdr}(s,t)$ is a conservative point estimate of $\fdr(s,t)$ under any configurations as long as those null hypotheses are uniformly distributed.  

\begin{thm}\label{thm:conservative}
Suppose the P-values corresponding to true null hypotheses are uniformly distributed in $[0,1]$.  Then, for any fixed $s \le t$, 
\beqn
\E[\widehat{\fdr}(s,t)] \ge \fdr(s,t).
\eeq
\end{thm}

Large scale multiple testing appears in many areas of applications, where $n$ is typically of the order of tens or hundreds of thousand.  This has led to the consideration of an asymptotic setting where $n$ tends to infinity \citep{storey2002direct, storey2004strong, genovese2002operating}. 
In detail, the asymptotic framework we consider requires the almost sure pointwise convergence of the empirical distribution of the null P-values and of the empirical distribution of the non-null P-values, or in formula,
\beq \label{pointlimit}
\lim_{n \to \infty} \frac{V(s,t)}{n_0} = t-s \quad \text{and} \quad \lim_{n \to \infty} \frac{S(s,t)}{n_1} = G(t) - G(s),
\eeq
almost surely for any fixed $0 \le s<t\le 1$, where $G$ is a continuous distribution function on the real line.  
We assume in addition that the following limit exists,
\beq \label{fraction}
\pi_0 := \lim_{n \to \infty} \frac{n_0}{n} \in (0,1), \quad \pi_1 := 1 - \pi_0.
\eeq
For the remaining results, we assume that Conditions \eqref{pointlimit}-\eqref{fraction} hold. Note that these conditions were also assumed in \citep{storey2004strong}.

\begin{rem}
This asymptotic framework generalizes the Bayesian model where the null hypotheses are true with probability $\pi_0$ and not true with probability $\pi_1$, and the null P-values are uniform in $[0,1]$ and the non-null P-values are $G$-distributed, corresponding to a mixture model where the P-values are iid with distribution function $\pi_0 t + \pi_1 G(t)$.
\end{rem}

Define 
\beqn
\overline{\fdr}^\infty(s,t) = \frac{t-s}{\pi_0 (t-s) + \pi_1 (G(t) - G(s))},
\eeq
which is the pointwise (almost sure) limit of $\widehat{\fdr}(s,t)$ under the above assumptions. Our next result shows that the scan procedure controls the FDR asymptotically.
Here, and everywhere else in the paper, $\alpha$ will denote the level at which the FDR is to be controlled.
We make the dependency of $(\hat\sigma, \hat\tau)$ on $n$ explicit, but note that other quantities, such as $\widehat\fdr, \fdp, \fdr$, also depend on $n$.

\begin{thm}\label{thm:FDRcontrol}
We have
\beqn
\limsup_{n \to \infty}\, \fdp(\hat\sigma_n, \hat\tau_n) \le \alpha, \quad \text{almost surely}, \quad \text{and} \quad
\limsup_{n \to \infty}\, \E\big[\fdp(\hat\sigma_n, \hat\tau_n)\big] \le \alpha.
\eeq
\end{thm}

\begin{rem}
In our notation, $\fdr(\hat\sigma_n, \hat\tau_n)$ is random and different from $\E\big[\fdp(\hat\sigma_n, \hat\tau_n)\big]$.  The latter is the FDR of the scan procedure with rejection region $[\hat\sigma_n, \hat\tau_n]$.
\end{rem}

We consider the maximization in \eqref{scan}, but based on $\overline{\fdr}^\infty$.  Indeed, let $\cA$ be the set of maximizers and $\delta$ the value of the following optimization problem
\beqn
\max\big\{t-s: \overline{\fdr}^\infty(s,t) \le \alpha\big\},
\eeq
or, equivalently, 
\beqn
\max\big\{t-s: G(t) - G(s) = \beta (t-s)\big\},
\eeq
where $\beta := \frac{1}{\pi_1} (\tfrac{1}{\alpha} - \pi_0)$.

\beq\label{A}
\begin{tabular}{p{0.8\textwidth}}
We assume that $\delta > 0$ and that there is $(s, t) \in \cA$ such that $u \mapsto \overline{\fdr}^\infty(u, t)$ strictly decreasing at $u = s$ or that $u \mapsto \overline{\fdr}^\infty(s, u)$ strictly increasing at $u = t$.
\end{tabular}
\eeq
The strict monotonicity condition is true, for example, if $G$ is concave on $[0,1]$, or more generally if $G$ is differentiable as satisfies $G'(s) \vee G'(t) < (G(t) - G(s))/(t-s)$ at some $(s,t) \in \cA$.

\begin{thm}\label{thm:limit}
If \eqref{A} holds, then, almost surely, any accumulation point of $(\hat\sigma_n, \hat\tau_n)$ belongs to $\cA$.
\end{thm}

\begin{rem}
This result is analogous to Theorem~1 in \citep{genovese2002operating}, which establishes a similar limit for the BH method under similar conditions.  
Specifically, they show that, almost surely, $\hat\tau_{\diamond, n}$ converges 
to 
\beq\label{BH-tau}
\delta_\diamond := \max\big\{ t : \overline\fdr^\infty(t) \le \alpha\},
\eeq 
with $\overline\fdr^\infty(0, t)$.
Alternatively, $\delta_\diamond$ may also be defined as the right-most solution to the equation $G(t) = \beta t$.
\end{rem}

\section{False non-discovery rate} \label{sec:FNR-results}
Having established that the scan procedure asymptotically controls the FDR at the desired level, we now turn to examining its false non-discovery rate (FNR).  We do so under the same asymptotic framework.  


\begin{thm}\label{thm:FNR}
If \eqref{A} holds, then
\beqn
\lim_{n \to \infty} \E\big[\fnp(\hat\sigma_n, \hat\tau_n)\big] = 1 - \beta \delta.
\eeq
\end{thm}
As could be anticipated from \thmref{limit}, the limiting value is the asymptotic FNR of any deterministic rule given by an interval $[s, t]$ with $(s, t) \in \cA$.

\begin{rem}
This result is analogous to Theorem~3 in \citep{genovese2002operating}, which establishes a similar limit for the BH method, specifically,
\beqn
\lim_{n\to\infty} \E\big[\fnp(\hat\tau_{\diamond, n})\big] = 1 - \beta \delta_\diamond.
\eeq 
\end{rem}

We now turn our attention to comparing the scan method and the BH method.  The following theorem provides some sufficient conditions under which the scan procedure outperforms the BH procedure.

\begin{thm}\label{thm:outperform}
Assume that $G$ is differentiable. 
If \eqref{A} holds, and in addition
\beq \label{property1}
G'(0) < G'(\delta_\diamond), 
\eeq
then the scan procedure has strictly smaller asymptotic FNR than the BH procedure.
\end{thm}

The BH procedure is known to be optimal in various ways under generalized Gaussian location models \citep{ariaschen2016distribution, rabinovich2017optimal}.  We therefore consider power-law location models.  
More specifically, we consider a mixture model where 
\beq\label{X}
X_1, \dots, X_n \iid \pi_0 \Psi(x) + \pi_1 \Psi(x - \mu),
\eeq 
where $\Psi$ is a continuous distribution on the real line and $\mu > 0$.  These are meant to represent the test statistics, whose large values weigh against their respective null hypotheses.  In particular, $\Psi$ is the null distribution and $\mu$ is the effect size.  The P-values are then computed as usual, meaning $P_i = \bar\Psi(X_i)$ where $\bar\Psi := 1 - \Psi$, and are seen to follow a mixture model
\beq\label{mixture}
P_1, \dots, P_n \iid \pi_0 t + \pi_1 G(t), \quad \text{where } G(t) := \bar\Psi(\bar\Psi^{-1}(t) - \mu).
\eeq

\begin{thm} \label{thm:propcondition}
Consider a mixture model \eqref{X} in the asymptotic defined by \eqref{fraction}.
Then the condition \eqref{pointlimit} holds.
Assume in addition that $\Psi$ has a density $\psi$ which can be taken to be strictly positive everywhere and such that $\psi(x) \to 0$ as $x \to \infty$ and $\psi(x) \sim x^{-\gamma-1} (\log x)^c$ as $x \to \infty$ for some $\gamma > 0$ and some $c \in \bbR$.  
Then there is $\mu_0 > 0$ (depending on $\Psi$ and $\beta$) such that \eqref{property1} holds for all $\mu > \mu_0$.  
\end{thm}

\begin{rem}
The result does not say anything about \eqref{A}, which is also required in \thmref{FNR}, but this condition is fulfilled except in pathological cases.
\end{rem}

\section{Numerical experiments}
\label{sec:numerics}

In this section, we perform simple simulations to see the performance of the BH and scan procedures on finite data. We consider the normal and Cauchy mixture models, as in \eqref{X}.

In the set of experiments, the sample size $n \in \{2, \dots, 8\} \times 10^3$. 
We draw $m = n(1-\eps)$ observations from the alternative distribution $\Psi(\cdot - \mu)$, and the other $n-m$ from the null distribution $\Psi$.  
Each situation is repeated 100 times and we report the average FDP and FNP for each procedure together with error bars.  The FDR control level was set at $\alpha = 0.10$. 

\subsection{Normal model}
In this model $\Psi$ is the standard normal distribution. 
We set $\eps = 0.05$ and $\mu = 4$. 
See \figref{normal}, where we have plotted $G$, the P-value distribution under the alternative defined in \eqref{mixture}.  This is a situation where $G$ is concave, so we expect the two methods to behave similarly.  This is confirmed numerically.  In fact, the scan procedure was observed, in these experiments, to coincide with the BH method.  (This does not happen at smaller signal-to-noise ratios, e.g., when $\mu$ is smaller.)
See \figref{performance-normal}, where we have plotted the FDP and FNP of both procedures.

\begin{figure}[H]
    \centering
    \includegraphics[scale = 0.4]{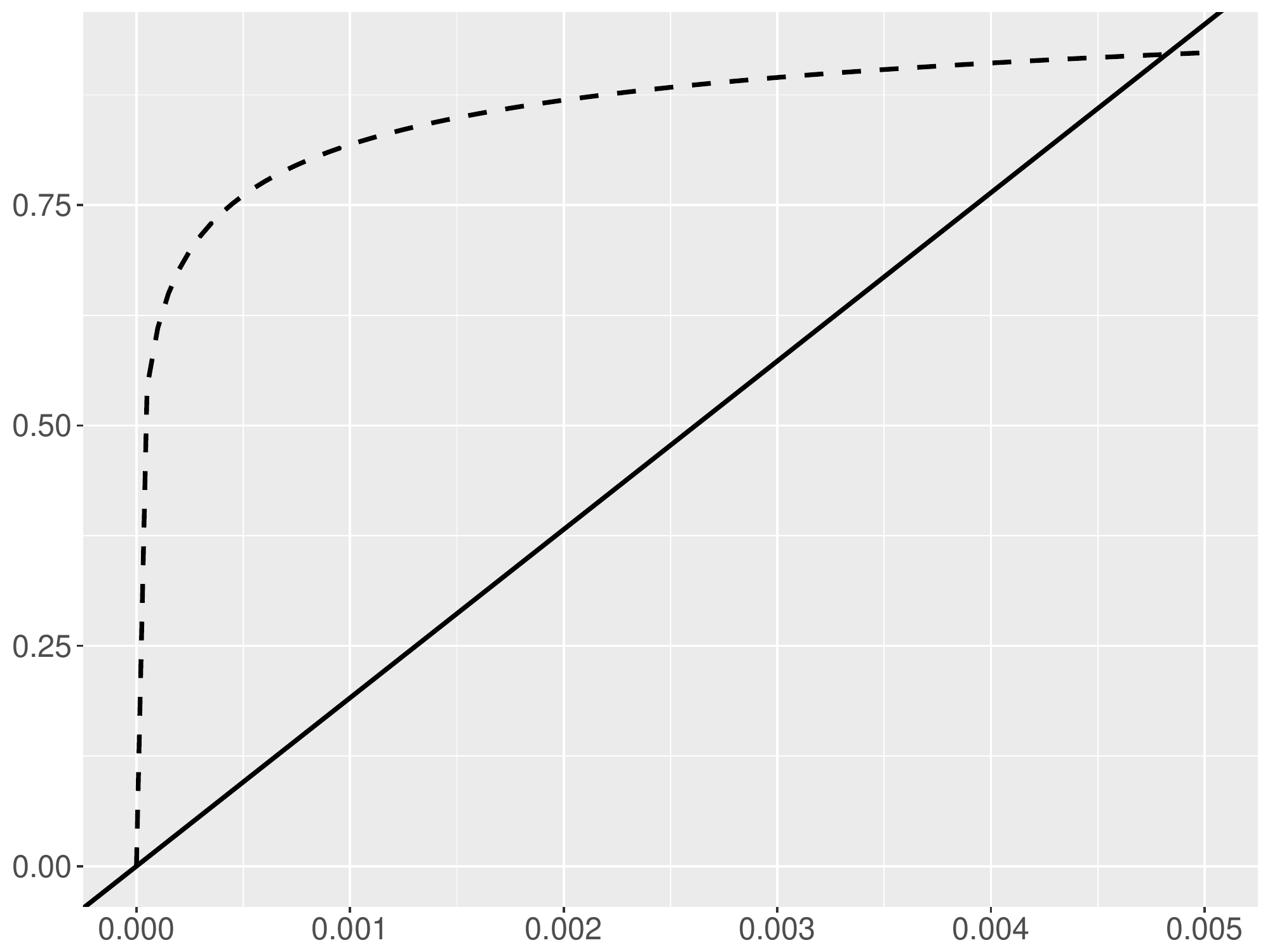}
    \caption{The alternative P-value distribution $G$ in the normal mixture model with $\eps = 0.05$ and $\mu = 4$ (solid black) and the line $y = \beta x$ (dashed black).}
    \label{fig:normal}
\end{figure}

\begin{figure}[t!]
\centering
\includegraphics[scale = 0.4]{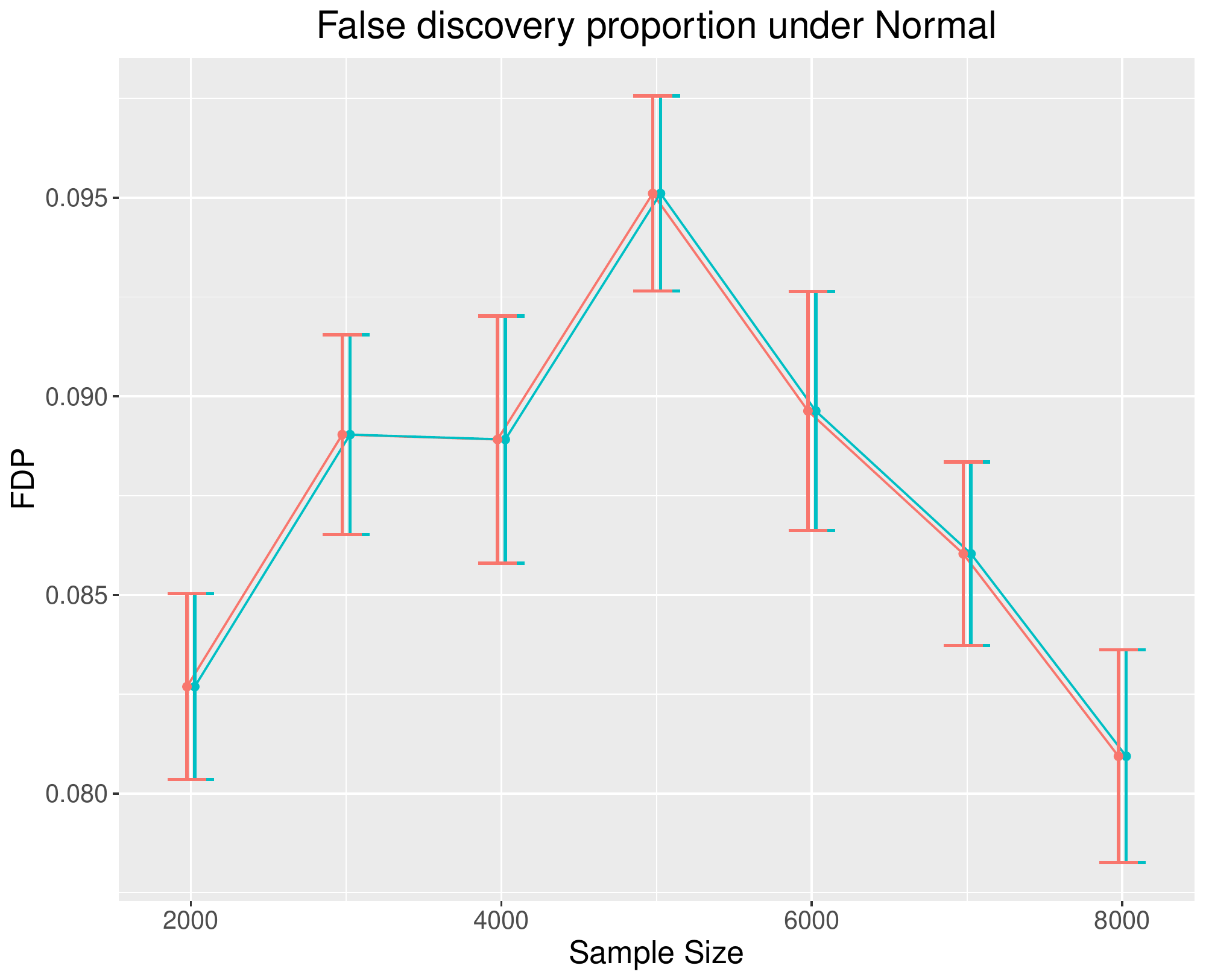}
\includegraphics[scale = 0.4]{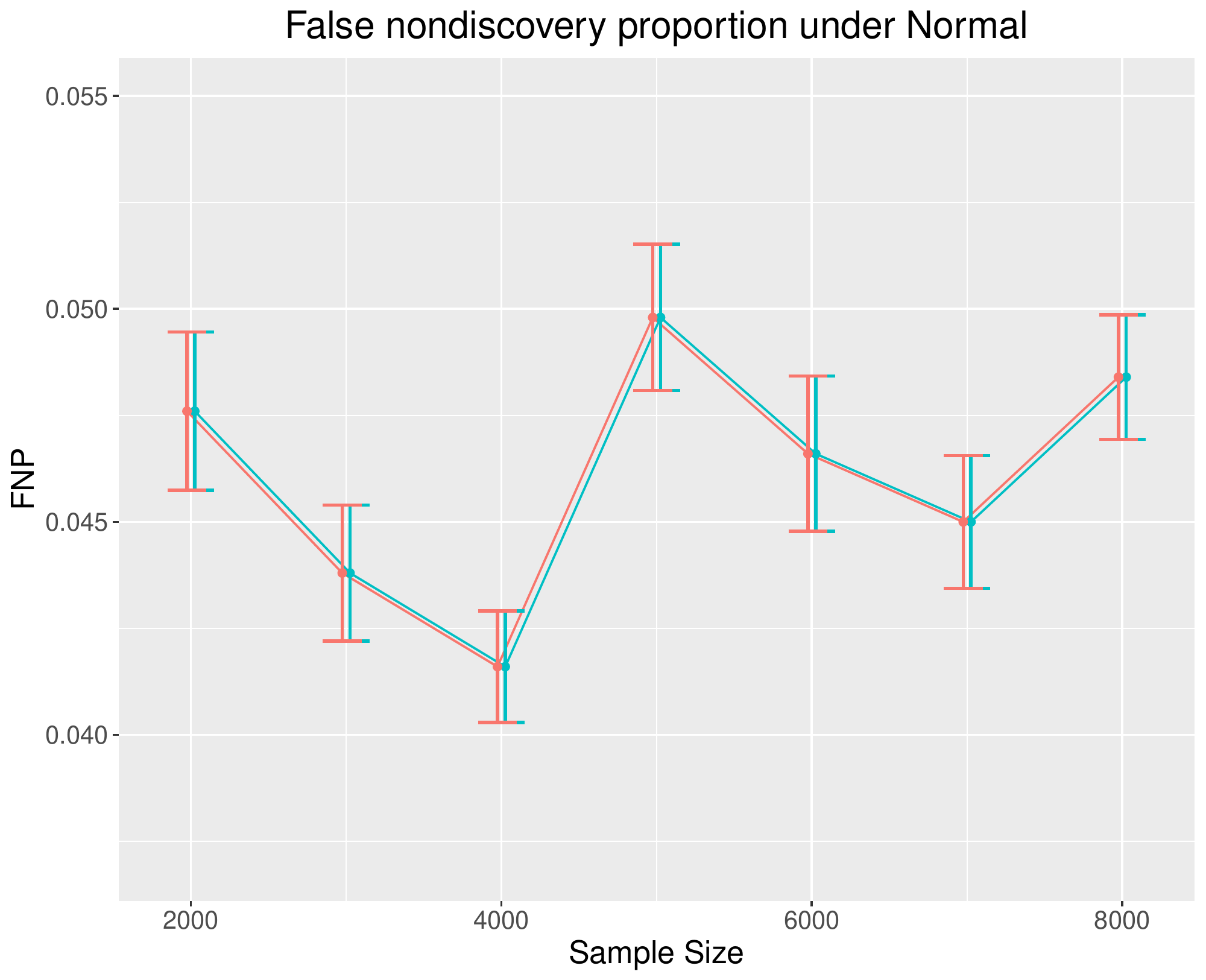}
\caption{FDP and FNP for the BH (red) and scan (blue) methods under normal mixture model.  The methods are essentially identical.  The FDR control was set at $\alpha = 0.10$. }
\label{fig:performance-normal}
\end{figure}

\subsection{Cauchy model}
In this model $\Psi$ is the Cauchy distribution. We set $\eps = 0.10$ and $\mu = 37$.  This choice of parameters leads to a model that satisfies the condition \eqref{property1} in \thmref{outperform}.  See \figref{cauchy} for an illustration. Therefore, here we expect the scan procedure to outperform the BH procedure.  This is confirmed in the numerical experiments.  
See \figref{performance-cauchy}, where we have plotted the FDP and FNP of both procedures.

\begin{figure}[t!]
    \centering
    \includegraphics[scale = 0.4]{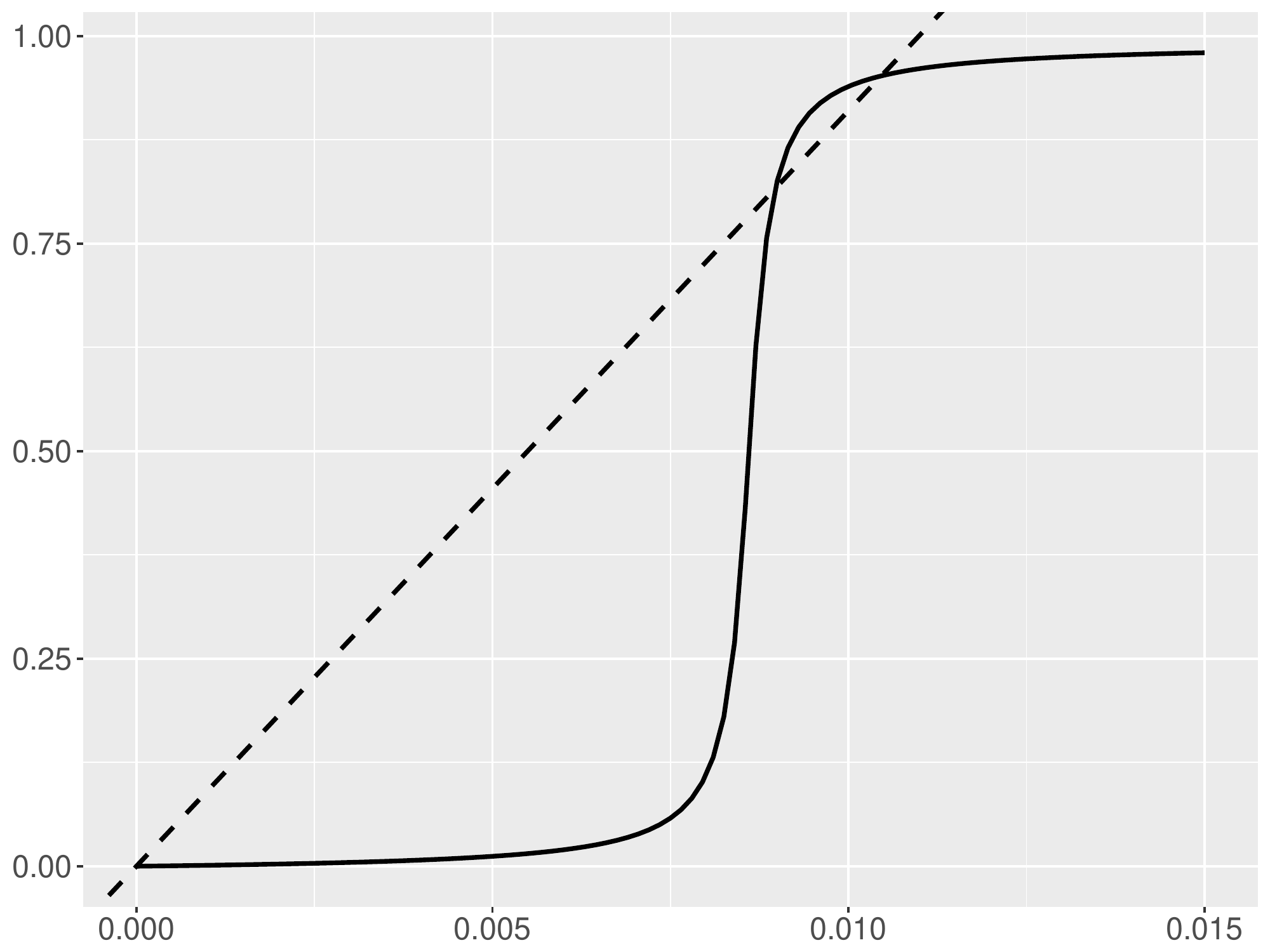}
    \caption{The alternative P-value distribution $G$ in the Cauchy mixture model with $\eps = 0.10$ and $\mu = 37$ (solid black) and the line $y = \beta x$ (dasned black). }    \label{fig:cauchy}
\end{figure}

\begin{figure}[t!]
    \centering
    \includegraphics[scale = 0.4]{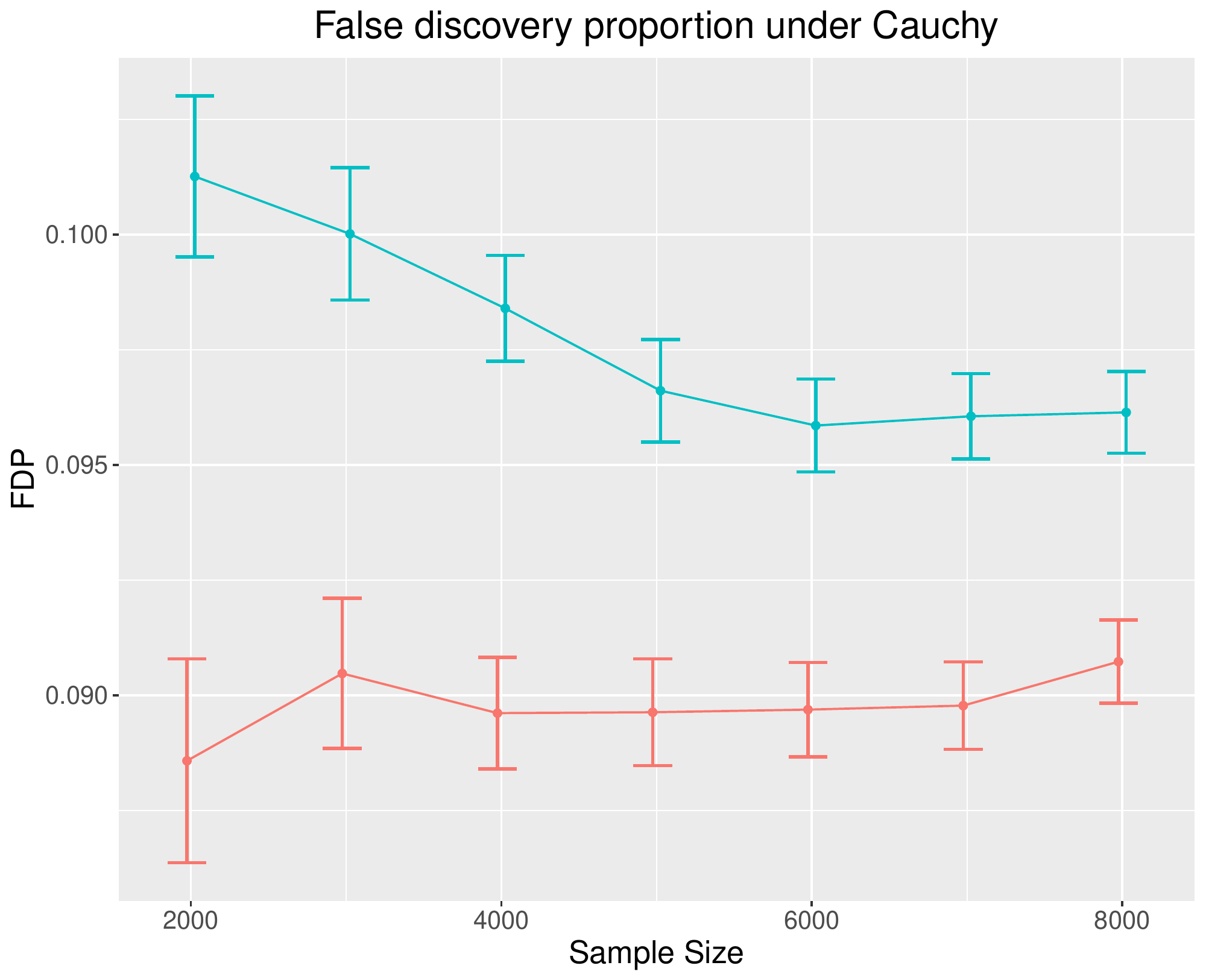}
    \includegraphics[scale = 0.4]{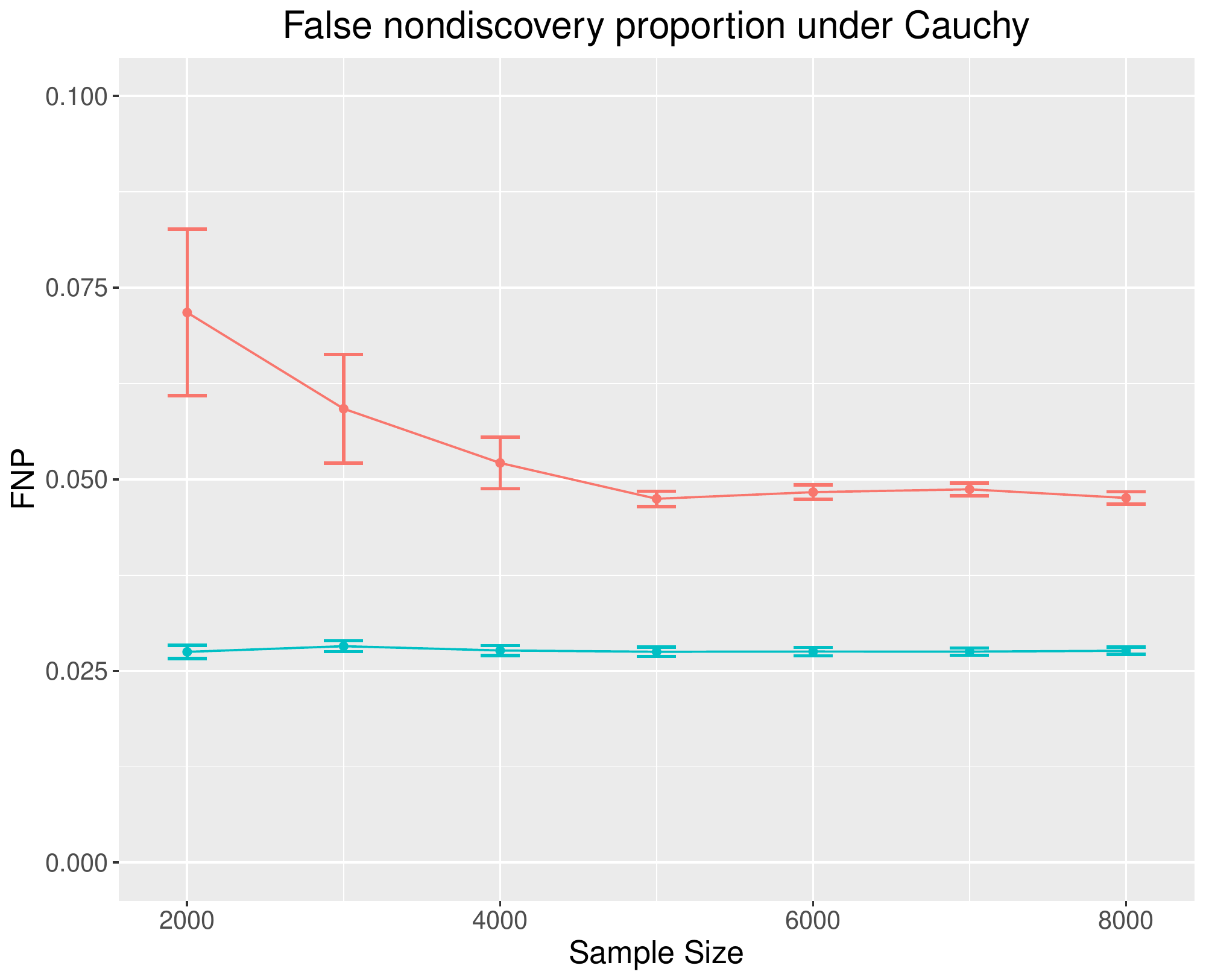}
\caption{FDP and FNP for the BH (red) and scan (blue) methods under Cauchy mixture model.  The methods are essentially identical.  The FDR control was set at $\alpha = 0.10$. }
    \label{fig:performance-cauchy}
\end{figure}

\section{Discussion} \label{sec:discussion}

\cite{genovese2002operating} argue that the BH method is not optimal among threshold procedures due to its being conservative in terms of FDR control.  We expect the same to be true of our scan procedure.  We could have pursued an improvement analogous to how the BH method was ameliorated in \citep{storey2002direct, benjamini2000adaptive} based on estimating the number of true null hypotheses ($n_0$ in our notation), but we chose not to do so for the sake of simplicity and focus.

We also want to mention that the present situation, where a scan method is found to improve upon a threshold method, has a parallel in the test of the global null hypothesis.  Indeed, continuing with the line of work coming out of \citep{ingster1997some, donoho2004higher}, we have recently considered the problem of detecting a sparse mixture and shown that threshold tests are inferior to scan tests in power-law models, although in a somewhat different asymptotic regime \citep{arias2018detection}.

\section{Proofs}
\label{sec:proofs}

We prove our results in this section.

\subsection{Proof of \thmref{conservative}} \label{sec:proof-conservative}
For any $s \le t$, 
\begin{align}
& \E \big [\widehat{\fdr}(s,t) - \fdr(s,t) \big ] \\
&= \E \bigg [\frac{n(t-s) - V(s,t)}{R(s,t) \vee 1} \bigg] \\
&\ge \E \bigg [\frac{n(t-s) - V(s,t)}{R(s,t)} \IND{V(s,t) \ge 1} \bigg]  \\
& = \E \bigg [ \E \bigg [\frac{n(t-s) - V(s,t)}{V(s,t) + S(s,t)} \IND{V(s,t) \ge 1}\ \bigg|\ S(s,t) \bigg] \bigg ] \\
& \ge \E \bigg [\frac{n(t-s) - \E \big [ V(s,t) \IND{V(s,t) \ge 1}\ \big|\ S(s,t) \big]}{\E \big [ V(s,t) \IND{V(s,t) \ge 1}\ \big|\ S(s,t) \big] + S(s,t)} \bigg]  \label{Jensen} \\ 
& = \E \bigg [\frac{n(t-s) - n_0(t-s)}{n_0(t-s) + S(s,t)}  \bigg] \ge 0, \label{expectation}
\end{align}
where \eqref{Jensen} follows from Jensen's inequality, based on the fact that $v \mapsto (a-v)/(b+v)$ is convex over $v \ge 0$, while \eqref{expectation} uses the fact that P-values are independent and uniform in $[0,1]$ under their null.  The fact that the very last expression is non-negative comes from the fact that $n_0 \le n$.

\subsection{Some preliminaries}

Henceforth, we assume that \eqref{pointlimit} and \eqref{fraction}.  Before proving our main results, we establish a few auxiliary lemmas.

\begin{lem}\label{lem:infty-conv}
For any fixed $d > 0$, almost surely,
\beq\label{infty-conv}
\lim_{n \to \infty} \sup_{t-s \ge d} \big| \widehat{\fdr}(s,t) - \overline{\fdr}^\infty(s,t) \big| = 0.
\eeq
\end{lem}

\begin{proof}
As is well-known, the pointwise convergences that we assume, namely \eqref{pointlimit}, imply uniform convergences, so that, together with \eqref{fraction}, we have 
\beq \label{GC-null}
\lim_{n \to \infty} \sup_{s \le t} \bigg| \frac{V(s,t)}{n} - \pi_0(t-s) \bigg | = 0,
\eeq
\beq \label{GC-alt}
\lim_{n \to \infty} \sup_{s \le t} \bigg| \frac{S(s,t)}{n} - \pi_1(G(t) - G(s))\bigg | = 0,
\eeq
almost surely. 
Combining these also yields
\beq \label{GC-mixture}
\lim_{n \to \infty} \sup_{s \le t} \bigg| \frac{R(s,t)}{n} - \big\{\pi_0(t-s) + \pi_1(G(t) - G(s))\big\}\bigg | = 0,
\eeq
When $t-s \ge d$, we have 
\beqn
\pi_0 (t-s) + \pi_1 (G(t)-G(s)) \ge \pi_0 d > 0,
\eeq
and it is thus straightforward to show that 
\beqn
\lim_{n \to \infty} \sup_{t-s \ge d} \bigg| \frac{n(t-s)}{R(s,t) \vee 1} - \frac{t-s}{\pi_0 (t-s) + \pi_1(G(t) - G(s)} \bigg| = 0, 
\eeq
almost surely, which establishes our claim.
\end{proof}

\begin{lem}\label{lem:uniform}
Almost surely,
\beqn
\lim_{n \to \infty} \inf_{s \le t} \big\{\widehat{\fdr}(s,t) - \fdp(s,t) \big\} \ge 0.
\eeq
\end{lem}

\begin{proof}
For the first part, we have
\beqn
\widehat{\fdr}(s,t) - \fdp(s,t) = \frac{n(t-s) - V(s,t)}{R(s,t) \vee 1} \ge 0 \iff t-s - \frac{V(s,t)}n \ge 0,
\eeq
so that we only need to prove that
\beqn
\lim_{n \to \infty} \inf_{s \le t} \big \{t-s - V(s,t)/n\big\} \ge 0.
\eeq
But this simply comes from \eqref{GC-null} and \eqref{fraction}.
\end{proof}

\begin{lem}\label{lem:point-lb}
If \eqref{A} holds, then, almost surely,
\beq\label{point-lb}
\liminf_{n \to \infty}  \big\{\hat\tau_n - \hat\sigma_n \big\} \ge \delta.
\eeq
\end{lem}

\begin{proof}
Let $(s, t)$ be as in \eqref{A}, with (for example) $u \mapsto \overline{\fdr}^\infty(u,t)$ strictly decreasing at $u = s$.  Then there is $\eps > 0$ such that $\overline{\fdr}^\infty(u,t) < \overline{\fdr}^\infty(s,t)$ when $s < u \le s + \eps$.  
With probability one, $\widehat{\fdr}(u,t)$ converges to $\overline{\fdr}^\infty(u,t)$, and when this is the case, $\widehat{\fdr}(u,t) \le \alpha$ for $n$ sufficiently large, then implying that $t-u \le \hat\tau_n - \hat\sigma_n$ by definition in \eqref{scan}.  Hence, we have shown that for any such $u$, $\liminf_{n \to \infty}  \big\{\hat\tau_n - \hat\sigma_n \big\} \ge t - u$ almost surely, and we conclude by letting $u \searrow s$.  (Recall that $\delta = t -s$ for any $(s,t) \in \cA$.)
\end{proof}


\subsection{Proof of \thmref{FDRcontrol}}
For the first part, using \lemref{uniform}, we have
\beqn
\liminf_{n \to \infty}  \big [\widehat{\fdr}(\hat\sigma_n, \hat\tau_n) - \fdp(\hat\sigma_n, \hat\tau_n) \big ] \ge 0,
\eeq
almost surely, and we conclude with \eqref{fdr-alpha}.

The second part just follows from the first part and Fatou's lemma.

\subsection{Proof of \thmref{limit}}
With probability one, a realization satisfies \eqref{infty-conv} with $d = \delta/2$, and \eqref{point-lb}.
Consider such a realization and let $(s^*, t^*)$ be an accumulation point of $(\hat\sigma_n, \hat\tau_n)$.  

Because \eqref{point-lb} holds, we have $t^* - s^* \ge \delta$.  

We also have $\hat\tau_n - \hat\sigma_n \ge d$, eventually, and because \eqref{infty-conv} holds, this implies that
\beqn
\lim_{n \in \cN} \widehat{\fdr}(\hat\sigma_n, \hat\tau_n) - \overline{\fdr}^\infty (\hat\sigma_n, \hat\tau_n) = 0.
\eeq
Together with \eqref{fdr-alpha}, we thus have
\beqn
\limsup_{n \to \infty} \overline{\fdr}^\infty (\hat\sigma_n, \hat\tau_n) \le \alpha.
\eeq
By continuity of $\overline{\fdr}^\infty$, this implies that $\overline{\fdr}^\infty(s^*, t^*) \le \alpha$, in turn implying that $t^* - s^* \le \delta$.  

We have thus established that $(s^*, t^*)$ satisfies $t^* - s^* = \delta$ and $\overline{\fdr}^\infty(s^*, t^*) \le \alpha$, and therefore $(s^*, t^*)$ belongs to $\cA$ by definition.

\subsection{Proof of \thmref{FNR}}
By definition of $S$ in \tabref{outcomes}, we have
\beqn
\fnp(\hat\sigma_n, \hat\tau_n) = 1 - S(\hat\sigma_n, \hat\tau_n)/n_1.
\eeq
By \eqref{GC-alt} and \thmref{limit} together with the fact that $G(t) - G(s) = \beta \delta$ for any $(s, t) \in \cA$, almost surely,
\beqn
S(\hat\sigma_n, \hat\tau_n)/n_1 \to \beta \delta.
\eeq
We thus have, almost surely,
\beqn
\fnp(\hat\sigma_n, \hat\tau_n) \to 1 - \beta \delta,
\eeq
and we conclude using the Dominated Convergence theorem.

\subsection{Proof of \thmref{outperform}}
Adapting the proof of \thmref{FNR}, we can establish an analogous result for the BH method, specifically,
\beq 
\E[\fnp(\hat\tau_{\diamond, n})] \to 1 - \beta \delta_\diamond,
\eeq 
almost surely, where $\delta_\diamond$ was defined in \eqref{BH-tau}.
Therefore, to compare the asymptotic FNR of the scan and the BH procedures, we need to compare $\delta$ and $\delta_\diamond$.

Define $A_0 = \arctan(G'(0))$, $A_\diamond = \arctan(G'(\delta_\diamond))$, and $B = \arctan(\beta)$. Apparently, we have $0 \le A_0, A_\diamond, B \le \frac{\pi}{2}$.
By the fact that $\delta_\diamond$ is the right-most solution to $G(t) = \beta t$, we have $G'(\delta_\diamond) \le \beta$.  Hence, $\phi := B - A_\diamond \ge 0$.
This, coupled with \eqref{property1}, implies that $G'(0) < \beta$, so that $\theta := B - A_0 > 0$.

Let $\cL_0$ denote the line with slope $\beta$ passing through the origin, and for $d \ge 0$, let $\cL_d$ denote the line parallel to $\cL_0$ at a distance $d$ below $\cL_0$.  
See \figref{property} for an illustration.  
Because $G'(0) < \beta$ and $G(\delta_\diamond) = \beta \delta_\diamond$, and by continuity of $G$, the graph of $G$ intersects $\cL_0$ at least twice.  Choosing $d$ small enough, it is therefore also the case that $G$ intersects $\cL_d$ at least twice.  Let $s_d$ and $t_d$ denote the horizontal coordinates of the leftmost and rightmost intersection points, respectively.  Note that $s_d \to 0$ as $d \to 0$ by that fact that $G'(0) < \beta$, and $t_d \to \delta_\diamond$ by the fact that $\delta_\diamond$ is the right-most solution to $G(t) = \beta t$.
\begin{figure}[t!]
    \centering
	\includegraphics[width = 7cm, height= 6cm]{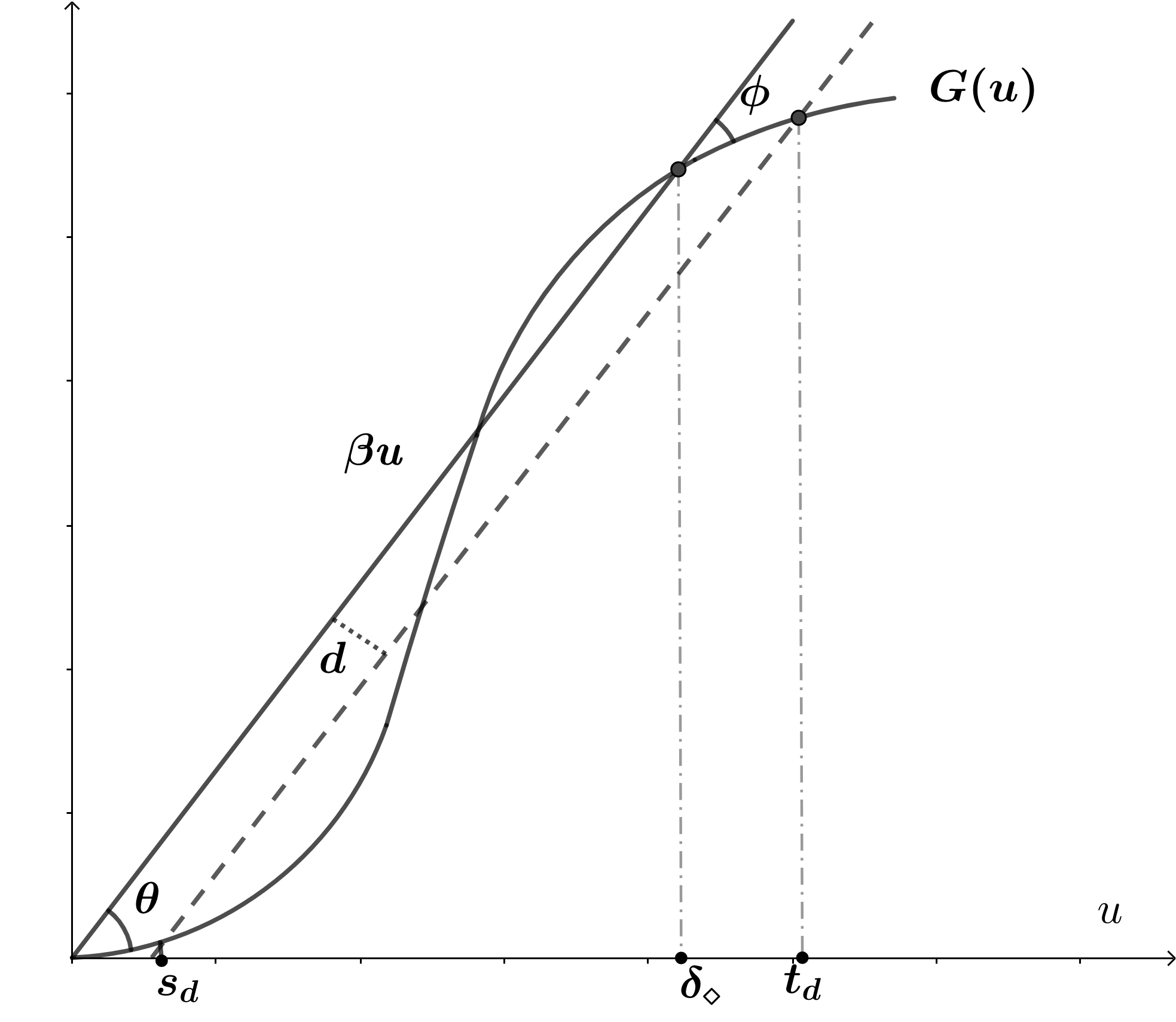}
	\hspace*{.5in}
	\caption{Example which satisfies Condition \ref{property1} in \thmref{outperform}.}
	\label{fig:property}	
\end{figure}
Moreover, as $d\to 0$, by simple geometry arguments, we have
\beqn
s_d \sim \sin(A_0) \cdot \frac{d}{\sin (\theta)}, \quad
t_d - \delta_\diamond \sim \sin(A_\diamond) \cdot \frac{d}{\sin(\phi)}.
\eeq
Since $G'(0) < G'(\delta_\diamond)$, we have $\theta > \phi \ge 0$ and also $A_0 < A_\diamond$.
It follows that, for $d$ small enough, $\delta_\diamond < t_d - s_d$.
Due to the fact, by construction, 
\beqn
\frac{G(t_d)-G(s_d)}{t_d-s_d} = \beta,
\eeq
we have $t_d-s_d \le \delta_\diamond$, by definition of the latter.

\subsection{Proof of \thmref{propcondition}}
The Law of Large Numbers implies that the condition \eqref{pointlimit} holds.  We thus turn to the remaining of the statement. 

We note that $G(t) = \bar{\Psi}(\bar{\Psi}^{-1}(t) - \mu)$ is differentiable on $(0,1)$, with derivative
\beqn
G'(t) = \frac{\psi(\bar{\Psi}^{-1}(t) - \mu)}{\psi(\bar{\Psi}^{-1}(t))}.
\eeq

As $t \to 0$, we have $\bar{\Psi}^{-1}(t) \to \infty$, and by the fact that for all $a \in \bbR$, $\psi(x - a) \sim \psi(x)$ as $x \to \infty$, we have that $G'$ is differentiable at $0$, with derivative $G'(0) = 1$.

We also have that, for any fixed $t$, $G'(t) \to 0$ as $\mu \to \infty$, due to the fact that $\psi(x) \to 0$ as $x \to -\infty$.  Hence, $\delta_\diamond \to 0$ as $\mu \to \infty$.  Let $x_\diamond = \bar\Psi^{-1}(\delta_\diamond)$, so that $x_\diamond \to \infty$ as $\mu \to \infty$.  Because $G(\delta_\diamond) = \beta \delta_\diamond$, we have $\bar\Psi(x_\diamond - \mu) = \beta \bar\Psi(x_\diamond)$.  
Because the right-hand side tends to 0, we must have $x_\diamond - \mu \to \infty$ as $\mu \to \infty$.  
Then using the fact that $\bar\Psi(x) \sim \frac1\gamma x^{-\gamma} (\log x)^c$ as $x \to \infty$, we must have
\beqn
\frac1\gamma (x_\diamond -\mu)^{-\gamma} (\log(x_\diamond -\mu))^c \sim \beta \frac1\gamma x_\diamond^{-\gamma} (\log x_\diamond)^c,
\eeq
or equivalently,
\beqn
(1 -\mu/x_\diamond)^{-\gamma} \bigg(\frac{\log(x_\diamond -\mu)}{\log x_\diamond}\bigg)^c \to \beta,
\eeq
as $\mu \to \infty$.  This is seen to imply that $x_\diamond \sim a \mu$, where $a := (1 - \beta^{-1/\gamma})^{-1}$.  Note that $a > 1$.  We then have, as $\mu \to \infty$,
\beqn
G'(\delta_\diamond) = \frac{\psi(x_\diamond - \mu)}{\psi(x_\diamond)} \sim \frac{(x_\diamond - \mu)^{-\gamma-1} (\log(x_\diamond - \mu))^c}{x_\diamond^{-\gamma-1} (\log x_\diamond)^c} \to \Big(\frac{a}{a-1}\Big)^{\gamma + 1} > 1.
\eeq

We conclude that, for $\mu$ large enough, $G'(\delta_\diamond) > 1 = G'(0)$.


\bibliographystyle{chicago}
\bibliography{ref}

\begin{thebibliography}{}

\bibitem[\protect\citeauthoryear{Arias-Castro and Chen}{Arias-Castro and
  Chen}{2016}]{ariaschen2016distribution}
Arias-Castro, E. and S.~Chen (2016).
\newblock Distribution-free multiple testing.
\newblock {\em arXiv preprint arXiv:1604.07520\/}.

\bibitem[\protect\citeauthoryear{Arias-Castro and Ying}{Arias-Castro and
  Ying}{2018}]{arias2018detection}
Arias-Castro, E. and A.~Ying (2018).
\newblock Detection of sparse mixtures: Higher criticism and scan statistic.
\newblock {\em arXiv preprint arXiv:1802.08715\/}.

\bibitem[\protect\citeauthoryear{Benjamini and Heller}{Benjamini and
  Heller}{2007}]{benjamini2007false}
Benjamini, Y. and R.~Heller (2007).
\newblock False discovery rates for spatial signals.
\newblock {\em Journal of the American Statistical Association\/}~{\em
  102\/}(480), 1272--1281.

\bibitem[\protect\citeauthoryear{Benjamini and Hochberg}{Benjamini and
  Hochberg}{1995}]{benjamini1995controlling}
Benjamini, Y. and Y.~Hochberg (1995).
\newblock Controlling the false discovery rate: A practical and powerful
  approach to multiple testing.
\newblock {\em Journal of the Royal Statistical Society. Series B
  (Methodological)\/}~{\em 57\/}(1), 289--300.

\bibitem[\protect\citeauthoryear{Benjamini and Hochberg}{Benjamini and
  Hochberg}{2000}]{benjamini2000adaptive}
Benjamini, Y. and Y.~Hochberg (2000).
\newblock On the adaptive control of the false discovery rate in multiple
  testing with independent statistics.
\newblock {\em Journal of educational and Behavioral Statistics\/}~{\em
  25\/}(1), 60--83.

\bibitem[\protect\citeauthoryear{Caldas~de Castro and Singer}{Caldas~de Castro
  and Singer}{2006}]{caldas2006controlling}
Caldas~de Castro, M. and B.~H. Singer (2006).
\newblock Controlling the false discovery rate: a new application to account
  for multiple and dependent tests in local statistics of spatial association.
\newblock {\em Geographical Analysis\/}~{\em 38\/}(2), 180--208.

\bibitem[\protect\citeauthoryear{Chi}{Chi}{2007}]{chi2007performance}
Chi, Z. (2007).
\newblock On the performance of fdr control: constraints and a partial
  solution.
\newblock {\em The Annals of Statistics\/}, 1409--1431.

\bibitem[\protect\citeauthoryear{Donoho and Jin}{Donoho and
  Jin}{2004}]{donoho2004higher}
Donoho, D. and J.~Jin (2004).
\newblock Higher criticism for detecting sparse heterogeneous mixtures.
\newblock {\em The Annals of Statistics\/}~{\em 32\/}(3), 962--994.

\bibitem[\protect\citeauthoryear{Genovese and Wasserman}{Genovese and
  Wasserman}{2002}]{genovese2002operating}
Genovese, C. and L.~Wasserman (2002).
\newblock Operating characteristics and extensions of the false discovery rate
  procedure.
\newblock {\em Journal of the Royal Statistical Society: Series B (Statistical
  Methodology)\/}~{\em 64\/}(3), 499--517.

\bibitem[\protect\citeauthoryear{Genovese and Wasserman}{Genovese and
  Wasserman}{2004}]{genovese2004stochastic}
Genovese, C. and L.~Wasserman (2004).
\newblock A stochastic process approach to false discovery control.
\newblock {\em Annals of Statistics\/}, 1035--1061.

\bibitem[\protect\citeauthoryear{Ingster}{Ingster}{1997}]{ingster1997some}
Ingster, Y.~I. (1997).
\newblock Some problems of hypothesis testing leading to infinitely divisible
  distributions.
\newblock {\em Mathematical Methods of Statistics\/}~{\em 6\/}(1), 47--69.

\bibitem[\protect\citeauthoryear{Naus}{Naus}{1965}]{naus1965distribution}
Naus, J.~I. (1965).
\newblock The distribution of the size of the maximum cluster of points on a
  line.
\newblock {\em Journal of the American Statistical Association\/}~{\em
  60\/}(310), 532--538.

\bibitem[\protect\citeauthoryear{Pacifico, Genovese, Verdinelli, and
  Wasserman}{Pacifico et~al.}{2007}]{pacifico2007scan}
Pacifico, M.~P., C.~Genovese, I.~Verdinelli, and L.~Wasserman (2007).
\newblock Scan clustering: A false discovery approach.
\newblock {\em Journal of Multivariate Analysis\/}~{\em 98\/}(7), 1441--1469.

\bibitem[\protect\citeauthoryear{Perone~Pacifico, Genovese, Verdinelli, and
  Wasserman}{Perone~Pacifico et~al.}{2004}]{perone2004false}
Perone~Pacifico, M., C.~Genovese, I.~Verdinelli, and L.~Wasserman (2004).
\newblock False discovery control for random fields.
\newblock {\em Journal of the American Statistical Association\/}~{\em
  99\/}(468), 1002--1014.

\bibitem[\protect\citeauthoryear{Picard, Reynaud-Bouret, and Roquain}{Picard
  et~al.}{2017}]{picard2017continuous}
Picard, F., P.~Reynaud-Bouret, and E.~Roquain (2017).
\newblock Continuous testing for poisson process intensities: A new perspective
  on scanning statistics.
\newblock {\em arXiv preprint arXiv:1705.08800\/}.

\bibitem[\protect\citeauthoryear{Rabinovich, Ramdas, Jordan, and
  Wainwright}{Rabinovich et~al.}{2017}]{rabinovich2017optimal}
Rabinovich, M., A.~Ramdas, M.~I. Jordan, and M.~J. Wainwright (2017).
\newblock Optimal rates and tradeoffs in multiple testing.
\newblock {\em arXiv preprint arXiv:1705.05391\/}.

\bibitem[\protect\citeauthoryear{Roquain}{Roquain}{2011}]{roquain2011type}
Roquain, E. (2011).
\newblock Type i error rate control in multiple testing: a survey with proofs.
\newblock {\em Journal de la Soci\'et\'e Fran\c{c}aise de Statistique\/}~{\em
  152\/}(2), 3--38.

\bibitem[\protect\citeauthoryear{Siegmund, Zhang, and Yakir}{Siegmund
  et~al.}{2011}]{siegmund2011false}
Siegmund, D., N.~Zhang, and B.~Yakir (2011).
\newblock False discovery rate for scanning statistics.
\newblock {\em Biometrika\/}~{\em 98\/}(4), 979--985.

\bibitem[\protect\citeauthoryear{Storey}{Storey}{2002}]{storey2002direct}
Storey, J.~D. (2002).
\newblock A direct approach to false discovery rates.
\newblock {\em Journal of the Royal Statistical Society: Series B (Statistical
  Methodology)\/}~{\em 64\/}(3), 479--498.

\bibitem[\protect\citeauthoryear{Storey, Taylor, and Siegmund}{Storey
  et~al.}{2004}]{storey2004strong}
Storey, J.~D., J.~E. Taylor, and D.~Siegmund (2004).
\newblock Strong control, conservative point estimation and simultaneous
  conservative consistency of false discovery rates: a unified approach.
\newblock {\em Journal of the Royal Statistical Society: Series B (Statistical
  Methodology)\/}~{\em 66\/}(1), 187--205.

\end{thebibliography}

\end{document}